\documentclass[11pt]{amsart}
\usepackage{amssymb,amsmath,txfonts}
\newtheorem{theorem}{Theorem}
\newtheorem{prop}{Proposition}
\newtheorem{lemma}{Lemma}
\newtheorem{claim}{Claim}
\newtheorem{remark}{Remark}

\numberwithin{equation}{section}

\author{Gang Liu}
\address{Department of Mathematics\\University of Minnesota\\Minneapolis, MN 55455}
\email{liuxx895@math.umn.edu}
\title[Local volume comparison]{Local volume comparison for K\"{a}hler manifolds}
\date{}
\begin{document}
\begin{abstract}On K\"ahler manifolds with Ricci curvature lower bound, assuming the real analyticity of the metric,
we establish a sharp relative volume comparison theorem for small balls.
The model spaces being compared to are complex space forms, i.e, K\"ahler manifolds with constant holomorphic sectional curvature.
Moreover, we give an example showing that on K\"{a}hler manifolds, the pointwise Laplacian comparison theorem does not hold when the Ricci curvature is bounded from below.
\end{abstract}
\maketitle

\section{\bf{Introduction}}
Comparison theorems are fundamental tools in geometric analysis. They are vital in the estimates of the spectrums, heat kernels and the Sobolev constants.
The classical Bishop-Gromov's relative volume comparison theorem [1][3][4] in Riemannian geometry is the following:
\begin{theorem}
 Let $M^n$ be a complete Riemannian manifold such that $Ric \geq (n-1)K$. For any $p \in M$ and $0 < a < b$, the volume of
geodesic balls satisfy
$$\frac{Vol(B_{p}(b))}{Vol(B_{p}(a))}\leq \frac{Vol(B_{M_K}(b))}{Vol(B_{M_K}(a))},$$ where $M_K$ is the simply
connected real space form with sectional curvature $K$, $Vol(B_{M_K}(r))$ is the volume of the
geodesic ball in $M_K$ with radius $r$. The equality holds iff $B_p(b)$ is isometric to $B_{M_K}(b)$.
\end{theorem}

The key ingredient in theorem 1 is the Laplacian comparison theorem [2][6]:
\begin{theorem}
{Let $M^n$ be a complete Riemannian manifold with $Ric \geq (n-1)K$. Let $M_k$ be the simply connected
real space form with sectional
 curvature $K$.   Denote $r_M(x)$ to be distance function from $p$ to $x$ in $M$. Let $r_{M_k}$ be the distance
function on $M_k$. Then for any $x \in M$,
 $y \in M_k$ with $r_M(x) = r_{M_k}(y)$,  $$\Delta r_M(x) \leq \Delta r_{M_k}(y).$$}
\end{theorem}

The model spaces in above theorems are real space forms.
In the K\"ahler category, it is a natural question whether we can replace the model spaces
by K\"ahler models, i.e, complex space forms which are K\"ahler manifolds with constant holomorphic sectional curvature.
In [5], Li and Wang showed that when the bisectional curvature has a lower bound, both theorems above hold with K\"ahler models. So the question left is:
what can we get if we only assume the lower bound of the Ricci curvature? This note addresses with the local case. The main theorem is the following:

\bigskip

\begin{theorem}{Let $M^n(n= dim_{\mathbb{C}}M)$ be a K\"{a}hler manifold with real
analytic metric. Assume $Ric \geq K$ ($K$ is any real number). Given any point $p\in M$, there exists $r=r(p, M)>0$ such that
for any $0< a < b < r$,
the volume of geodesic balls satisfy
 $$\frac{Vol(B_{M^n}(p, b))}{Vol(B_{M^n}(p, a))}\leq \frac{Vol(B_{N_K}(b))}{Vol(B_{N_K}(a))},$$
where $N_K$ denotes the rescaled
complex space form with $Ric = K$, $\Delta_{N_K}r$ is the Laplacian of distance function on $N_K$.
The equality holds iff $M$ is locally isometric to $N_K$.}
\end{theorem}

\begin{remark}{ Theorem 3 is a local version of
Bishop-Gromov's relative volume comparison theorem on K\"ahler manifolds. However, one cannot directly
extend theorem 3 to any radius. A simple example is the product
 of $\mathbb{P}^1$ with the standard product metric.
Then the diameter is greater than that of the complex space form. This implies
when $r$ is large, the inequality in theorem 3 does not hold.}
\end{remark}
We can prove a result which is slightly stronger than theorem 3:
\begin{theorem}{Under the same assumption as in theorem 3,
there exists $r_0=r_0(p, M)>0$ such that for any $r < r_0$, the average Laplacian comparison holds:
 $$\frac{\int_{\partial B_p(r)}\Delta r}{A(\partial B_p(r))}\leq \Delta_{N_K}r(r),$$ where $\Delta_{N_K}r$ is the Laplacian of distance function on $N_K$.
Moreover, the equality holds iff $M$ is locally isometric to $N_K$.}
\end{theorem}

\begin{remark}{Theorem 4 is a local version of theorem 2 in the average sense. However, on K\"ahler manifolds
with Ricci curvature lower bound, the pointwise Laplacian comparison does not hold even locally(see section 6).}
\end{remark}

The idea of the proof of theorem 4 is very simple. We shall expand the area of the geodesic
sphere $A(\partial B_p(r))$ by power series, then compare the coefficients with that of the
rescaled complex space form. The computation is complicated since it involves the covariant derivatives of the curvature tensor
with arbitrary order.

This note is organized as follows:

In section 2, we state two propositions which demonstrate the relation between the
 derivatives of $A(\partial B_p(r))$ and covariant derivatives of the curvature tensor at $p$. Section 3 is the first part
of the proof of proposition 1. We shall estimate the derivatives of $A(\partial B_p(r))$ up to order $4$.
In the estimate
of the $4$th derivative, the K\"ahler condition is employed. The most important part is section 4. We use an
induction to prove proposition 1. Besides the routine computation, there are two technical lemmas(lemma 3 and
lemma 4) which simplify the computation of higher order covariant derivatives of the curvature tensor significantly. One should
note that the K\"ahler condition is essential in these two lemmas. We complete the proof of proposition 2 and theorem 4 in section 5. The last section is devoted to giving an example showing that the pointwise Laplacian comparing
with the complex space form does not necessarily hold if the complex dimension is greater or equal to 2.

\bigskip
\vskip.1in
\begin{center}
\bf  {\quad Acknowledgment}
\end{center}

The author is grateful to his advisor Professor Jiaping Wang for continuous encouragement and helpful discussions during the work.

\section{\bf{Basic set up}}

\bigskip
Throughout this note, for derivatives of functions of $r$, we are always evaluating at $r=0$.
Given a point $p$ on a K\"{a}hler manifold $M^n$,  fix a unit vector $e_0 \in T_{p}M$. Along the geodesic $l$ from $p$ with initial direction $e_0$, consider the Jacobian equation $J''=R(e_0,J)e_0$. Set up an orthonormal frame $\{e_k\}$ at $p$ such that $Je_{2i}=e_{2i+1}, Je_{2i+1}=-e_{2i}$ for $0\leq i\leq n-1$. Parallel transport the frame along the geodesic $l$. Consider the Jacobian field $J_u$ with initial value $J_u(0)=0, J'_u(0)=e_u$.

We may write
\begin{equation}\label{1.1}
J_u=J_u(r, e_0)=\sum\limits_{i=1}^{\infty}\sum\limits_{v=0}^{2n-1} r^{i}C_{u,i}^{v}e_{v}
\end{equation}
where $C_{u,i}^{v}$ are constants independent of $r$.
Denote $R_{e_0e_ue_0e_v}$ by $R_{uv}$ when $e_0$ is fixed.
  Plugging $(2.1)$ in the Jacobian equation, we get
 \begin{equation}
  \sum\limits_{i}\sum\limits_{v} i(i-1)r^{i-2}C^{v}_{u,i}e_v=\sum\limits_{k}\sum\limits_{w}r^{k}C_{u,k}^{w}R(e_0,e_w)e_0.
\end{equation}
  Along the geodesic $l$,
 $$R(e_0,e_w)e_0=\sum\limits_{s=0}^{2n-1}\sum\limits_{j=0}^{\infty} \frac{R^{(j)}_{sw}}{j!}e_sr^j$$
where $R^{(j)}_{sw}$ denotes the $j$th covariant derivative of $R_{sw}$ along $e_0$ at $p$.
  Inserting it in $(2.2)$, we get
$$\sum\limits_{i,v}i(i-1)r^{i-2}C^v_{u,i}e_v=\sum\limits_{k,j,w,s}r^{k+j}C^w_{u,k}\frac{R^{(j)}_{sw}}{j!}e_s.$$

Comparing coefficients, we obtain
\begin{equation}
  C^v_{u,i}=\sum\limits_{k+j=i-2,w}C^w_{u,k}\frac{R^{(j)}_{vw}}{j!i(i-1)}.
\end{equation}

 A simple iteration gives
$$C^v_{u, 1} = \delta^v_u; C^w_{u, 2}=0; C^v_{u, 3}=\sum\limits_wC^w_{u, 1}\frac{R_{vw}}{6}=\frac{R_{uv}}{6};$$
$$C^v_{u, 4} = \sum\limits_wC^w_{u, 1}\frac{R'_{vw}}{12}= \frac{R'_{vu}}{12};$$
$$C^v_{u, 5} = \sum\limits_w(C^w_{u, 1}\frac{R''_{vw}}{40}+C^w_{u, 3}\frac{R_{vw}}{20})= \frac{1}{120}(\sum\limits_s R_{us}R_{sv}+3R''_{uv}).$$

  Employing (2.1), we have

\begin{equation}
  J_u=re_u+\frac{r^3}{6}R_{uv}e_v+\frac{r^4}{12}R'_{uv}e_v+\frac{r^5}{120}(\sum\limits_s R_{us}R_{sv}+3R''_{uv})e_v+O(r^6).
  \end{equation}

Using $dA$ to denote the standard measure of the unit tangent bundle $UT_p(M)$ at $p$, via exponential map, we write $\int_{\partial B(p, r)}dA$ as $\int$.
Defining $$W = \frac{\int \sqrt{det<J_u,J_v>}}{r^{2n-1}},$$ we introduce two propositions as follows:

\begin{prop}{ Under the same condition as in theorem 4, if the derivatives of $W$ with order from $1$ to $(2m-1)(m \geq 1)$ are the same as that of the complex space form, we have

Conclusion $1:$

 If $m = 2$, $Ric = K$ at $p$.

If $m \geq 3$, then $R_{i\overline{j}k\overline{l}} = \frac{K}{n+1}(\delta_{ij}\delta_{kl} + \delta_{il}\delta_{jk})$ at $p$. Moreover, for any unit vectors $u, v, e_0 \in UT_p(M)$,  $R^{(\lambda)}_{uv}=0$ for $1\leq\lambda\leq m-3$ and $Ric^{(l)}(e_0, e_0)=0$ for $1\leq l\leq 2m-4$. The superscripts are orders of covariant derivatives along direction $e_0$.

Conclusion $2:$ $W^{(2m)}$ is less than or equal to that of the complex space form.}
\end{prop}

\begin{prop}{ Under the same condition as in theorem 4, if the derivatives of $W$ with order from $1$ to $(2m)(m \geq 1)$ are the same as the complex space form, $W^{(2m+1)} = 0$.}
\end{prop}

We divide the proof of proposition 1 into two parts:
 $m = 1, 2$ and $m \geq 3$.

\section{\bf{The proof of proposition 1: Part I}}
This section treats the case $m =1, 2$.
By (2.1), we have
\begin{equation}
\frac{<J_u,J_v>}{r^2}=\sum\limits_{i,j,w}r^{i+j-2}C^w_{u,i}C^w_{v,j}.
\end{equation}
By (2.4),
$$\frac{<J_u,J_u>}{r^2} = 1 + \frac{R_{uu}}{3}r^2 + \frac{R_{uu}'}{6}r^3 + (\frac{2}{45}\sum\limits_{s}R_{us}^2
  + \frac{1}{20}R_{uu}'')r^4+ O(r^5).$$
If $u \neq v$,
$$\frac{<J_u,J_v>}{r^2} = \frac{1}{3}R_{uv}r^2 + \frac{R_{uv}'}{6}r^3 + (\frac{2}{45}\sum\limits_{s}R_{us}R_{vs} + \frac{1}{20}R_{uv}'')r^4+ O(r^5).$$

Now use the above two expressions to see that
\begin{equation}
\begin{aligned}
\frac{det<J_u,J_v>}{r^{4n-2}}&= 1 + \frac{1}{3}\sum\limits_{u}R_{uu}r^2 +
\frac{1}{6}\sum\limits_{u}R_{uu}'r^3 + (\frac{2}{45}\sum\limits_{u,s}R_{us}^2
  + \frac{1}{20}\sum\limits_{u}R_{uu}''\\&+ \frac{1}{9}\sum\limits_{u < v}R_{uu}R_{vv} - \frac{1}{9}\sum\limits_{u < v}R_{uv}^2)r^4 + O(r^5).
 \end{aligned}
\end{equation}

Considering the identity $\sqrt{1+x} = 1 + \frac{1}{2}x - \frac{1}{8}x^2 + O(x^3)$,
we get

\begin{equation}
\begin{aligned}
\frac{\sqrt{det<J_u,J_v>}}{r^{2n-1}}& = 1 + \frac{1}{6}\sum\limits_{u}R_{uu}r^2 +
\frac{1}{12}\sum\limits_{u}R_{uu}'r^3 + (\frac{1}{45}\sum\limits_{u,s}R_{us}^2
  + \frac{1}{40}\sum\limits_{u}R_{uu}''\\&+ \frac{1}{18}\sum\limits_{u < v}R_{uu}R_{vv} - \frac{1}{18}\sum\limits_{u < v}R_{uv}^2 - \frac{1}{72}(\sum\limits_{u}R_{uu})^2)r^4 + O(r^5).
  \end{aligned}
  \end{equation}
Since $W = \frac{\int \sqrt{det<J_u,J_v>}}{r^{2n-1}}$, we find
\begin{displaymath}
\begin{aligned}
W'(0) = 0, W''(0)=-cs
\end{aligned}
\end{displaymath}
where $c$ is a positive constant depending only on $n$, $s$ is the scalar curvature at $p$.
Therefore $W''(0)$ is less than or equal to that of the complex space form. This proves proposition 1 for $m = 1$.

Now we consider $m=2$. According to the assumption of proposition 1, $W''$ is the same as that of the complex space form.
Therefore $s = nK$ at $p$. Since the Ricci curvature is bounded from below by $K$, $Ric = Kg$ at $p$.
By (3.3), it is simple to see that the $r^3$ coefficient of $W$ is $0$ by symmetry.
Thus to complete the proof for $m = 2$, we just need to show that the 4th derivative of $W$ is less than or equal to that of the complex space form.

We keep in mind that $Ric = Kg$ at $p$.
The $r^4$ coefficient of $W$ is
\begin{displaymath}
\begin{aligned}
c_4&=\int (\frac{1}{45}\sum\limits_{u,s}R_{us}^2
  + \frac{1}{40}\sum\limits_{u}R_{uu}''+ \frac{1}{18}\sum\limits_{u < v}R_{uu}R_{vv} - \frac{1}{18}\sum\limits_{u < v}R_{uv}^2 - \frac{1}{72}(\sum\limits_{u}R_{uu})^2)\\&=\frac{1}{360}\int(8\sum\limits_{u}R^2_{uu}+ 16\sum\limits_{u<v}R^2_{uv} + 9 \sum\limits_uR_{uu}''
+20\sum\limits_{u<v}R_{uu}R_{vv}\\&-20\sum\limits_{u<v}R^2_{uv}-5(\sum\limits_uR_{uu})^2)\\&=
\frac{1}{360}\int(-2\sum\limits_{u}R^2_{uu}+10(\sum\limits_uR_{uu})^2-4\sum\limits_{u<v}R^2_{uv} + 9 \sum\limits_uR_{uu}''-5(\sum\limits_uR_{uu})^2)\\&=\frac{1}{360}\int(9\sum\limits_{u}R_{uu}''- 4\sum\limits_{u < v}R_{uv}^2 - 2\sum\limits_{u}R_{uu}^2+5(\sum\limits_uR_{uu})^2).
\end{aligned}
\end{displaymath}

Note that the Ricci curvature attains the minimum $K$ at $p$, so $$\sum\limits_{u}R''_{uu} = -Ric''(e_0,e_0) \leq 0.$$
Therefore we have
\begin{equation}
\begin{aligned}
  c_4&=\frac{1}{360}\int(9\sum\limits_{u}R_{uu}''- 4\sum\limits_{u < v}R_{uv}^2 - 2\sum\limits_{u}R_{uu}^2+5K^2)
  \\& \leq -\frac{1}{360}\int(2\sum\limits_{u}R_{uu}^2-5K^2) \\& = -\frac{1}{360}\int(2\sum\limits_{u \neq 1}R_{uu}^2 + 2R_{11}^2-5K^2) \\&\leq -\frac{1}{360}\int(\frac{1}{n-1}(\sum\limits_{u \neq 1}R_{uu})^2+2R^2_{11}-5K^2)
  \\& = -\frac{1}{360}\int(\frac{1}{n-1}(Ric(e_0, e_0)+R_{11})^2+2R^2_{11}-5K^2)\\& = -\frac{1}{360}\int (\frac{1}{n-1}K^2+\frac{2}{n-1}KR_{11}+(\frac{1}{n-1}+2)R^2_{11}-5K^2)\\& \leq -\frac{1}{360}(\int \frac{1}{n-1}K^2+\frac{2}{n-1}K\int R_{11}+C_1(\int R_{11})^2- \int 5K^2)\\& = C_2K^2.
  \end{aligned}
\end{equation}

In the inequalities above, $C_1, C_2$ are constants depending only on $n$.

We explain the inequalities above. In the first inequality, we drop the two terms $\sum\limits_{u < v}R_{uv}^2$ and $\sum\limits_{u}R_{uu}''$.
In the second inequality, we apply Schwartz inequality for directions $e_u$ that are perpendicular to $e_1, e_0$. In the third inequality we use Schwartz inequality
$\int R^2_{11} \geq C(\int R_{11})^2$. We make use of the K\"ahler condition to obtain $\int R_{11} = C_3s = nC_3K$, where $C_3$ is a constant depending only on $n$. This explains the last equality.

The right hand side of (3.4) is exactly the case of the complex space form. Therefore when $W', W''$ are the same as the complex space form, $W^{(3)} = 0$ and $W^{(4)}$ is less than or equal to that of the complex space form. (3.4) becomes an equality if and only if the holomorphic sectional curvature is constant at $p$ and $Ric''(e_0, e_0) = 0$ for any $e_0 \in UT_pM$.
This completes the proof for $m = 2$.

  \bigskip

\section{\bf{The proof of proposition 1: Part II}}

   This section deals with the case $m \geq 3$. Denote $Ric^{(l)}(e_0, e_0)$ by $Ric^{(l)}$.
According to the assumption of proposition 1, the derivatives of $W$ with order from 1 to $(2m-1)$ are the same
as the complex space form. Follow results in the last section, the holomorphic sectional curvature is constant
at $p$ and $Ric''= 0$ for any $e_0$. That is to say, at $p$,
$$R_{i\overline{j}k\overline{l}} = \frac{K}{n+1}(\delta_{ij}\delta_{kl} + \delta_{il}\delta_{jk}), Ric'' = 0.$$

  Therefore, we proved conclusion 1 of proposition 1 for $m = 3$.

Now we use induction. Assuming conclusion 1 of proposition 1 holds for $k = m$, we shall prove that for $k = m + 1$.

\bigskip

\begin{claim}{ Under the hypothesis of the induction above, $C^v_{u,i}$($i \leq m$) are constants independent of the
 direction $e_0$. In fact, they are the same as that of the complex space form($C^v_{u,i}$ is defined in (2.1)).}
\end{claim}

\begin{proof}

 Claim 1 follows if we insert the induction hypothesis in (2.3).

\end{proof}

   Let us write
\begin{equation}
\begin{aligned}
   \frac{det<J_u,J_v>}{r^{4n-2}} = 1 + \sum\limits_{i=1}^{m-1}a_ir^i+\sum\limits_{j=m}^{2m}b_jr^j+O(r^{2m+1}).
  \end{aligned}
\end{equation}

Combining claim 1 with (3.1), we find that $a_i$ are constants independent of the direction $e_0$.
(3.1) also yields $C^v_{u,m+1} = C^u_{v,m+1}$ for all $u, v$.
Direct expansion of the determinant via (3.1) gives
\begin{equation}
\begin{aligned}
  b_{2m}=&\sum\limits_{u,v}(C^v_{u,m+1})^2+4\sum\limits_{u<v}C^u_{u,m+1}C^v_{v.m+1}+2\sum\limits_uC^u_{u,2m+1}-4\sum\limits_{u<v}C^v_{u,m+1}C^u_{v,m+1} \\&+\sum\limits_{i=1}^{m}C^v_{u,m+i}C_{i,m,u,v}+C_{0,m}
  \end{aligned}
\end{equation}
where $C_{i,m,u,v}$ and $C_{0,m}$ are all constants independent of the direction $e_0$.

Note also
\begin{equation}
\begin{aligned}
b_{m}= 2\sum\limits_{u}C^{u}_{u,m+1} + Constant.
\end{aligned}
\end{equation}

   Applying $\sqrt{1+x} = 1 + \frac{1}{2}x - \frac{1}{8}x^2 + \sum\limits_{k=3}^{\infty}\lambda_kx^k$($|x| < 1 $), we obtain

\begin{equation}
\begin{aligned}
  \frac{\sqrt{det<J_u,J_v>}}{r^{2n-1}}&= 1 + \frac{1}{2} (\sum\limits_{i=1}^{m-1}a_ir^i+\sum\limits_{j=m}^{2m}b_jr^j) - \frac{1}{8}(\sum\limits_{i=1}^{m-1}a_ir^i+\sum\limits_{j=m}^{2m}b_jr^j)^2\\& + \sum\limits_{k=3}^{\infty}\lambda_k (\sum\limits_{i=1}^{m-1}a_ir^i+\sum\limits_{j=m}^{2m}b_jr^j)^k+O(r^{2m+1}).
   \end{aligned}
\end{equation}

 \begin{lemma} { the $2m$th order coefficient of the expansion of $W$ is 
\begin{equation}
\begin{aligned}
   c_{2m}&=\int (\frac{1}{2}\sum\limits_{u,v}(C^v_{u,m+1})^2+2\sum\limits_{u<v}C^u_{u,m+1}C^v_{v,m+1}+\sum\limits_uC^u_{u,2m+1}\\&-2\sum\limits_{u<v}C^v_{u,m+1}C^u_{v,m+1} -\frac{1}{2}(\sum\limits_{u}C^{u}_{u,m+1})^2+\sum\limits_{i=1}^{m}C^v_{u,m+i}\widetilde{C}_{i,m,u,v})+\widetilde{C}_{0,m}
  \end{aligned}
  \end{equation}
  where $\widetilde{C}_{i,m,u,v}$ and $\widetilde{C}_{0,m}$ are constants independent of the direction $e_0$.}
\end{lemma}
 \begin{proof}

 It suffices to find out the contribution of each term in (4.4) to $c_{2m}$.
We keep in mind that coefficients $a_i$ in (4.1) are independent of $e_0$.

By (4.2), the contribution of term $1 + \frac{1}{2} (\sum\limits_{i=1}^{m-1}a_ir^i+\sum\limits_{j=m}^{2m}b_jr^j)$ to $c_{2m}$ is
\begin{equation}
\begin{aligned}
 &\int\frac{1}{2}\sum\limits_{u,v}(C^v_{u,m+1})^2+2\sum\limits_{u<v}C^u_{u,m+1}C^v_{v,m+1}+\sum\limits_uC^u_{u,2m+1}-2\sum\limits_{u<v}C^v_{u,m+1}C^u_{v,m+1}
 \\&+\frac{1}{2}(\sum\limits_{i=1}^{m}C^v_{u,m+i}C_{i,m,u,v}+C_{0,m}).
 \end{aligned}
  \end{equation}
The contribution of the term $- \frac{1}{8}(\sum\limits_{i=1}^{m-1}a_ir^i+\sum\limits_{j=m}^{2m}b_jr^j)^2$ to $c_{2m}$ is
\begin{equation}
\begin{aligned}
-\int(\frac{1}{8}b^2_{m}+ \sum\limits_{i=1}^{m}C^v_{u,m+i}p_{i,m,u,v})+p_{0,m}.
\end{aligned}
  \end{equation}
By (4.3), it could be written as
\begin{equation}
\begin{aligned}
-\int(\frac{1}{2}(\sum\limits_{u}C^{u}_{u,m+1})^2+ \sum\limits_{i=1}^{m}C^v_{u,m+i}p_{i,m,u,v})+p_{0,m}.
\end{aligned}
  \end{equation}
The contribution of $\sum\limits_{k=3}^{\infty}\lambda_k (\sum\limits_{i=1}^{m-1}a_ir^i+\sum\limits_{j=m}^{2m}b_jr^j)^k$ to $c_{2m}$ is
\begin{equation}
\begin{aligned}
  \int\sum\limits_{i=1}^{m}C^v_{u,m+i}q_{i,m,u,v}+q_{0,m}.
\end{aligned}
  \end{equation}

 In (4.7), (4.8), (4.9), $p_{i,m,u,v}$, $q_{i,m,u,v}$, $p_{0,m}$ and $q_{0,m}$ are all constants independent of the direction $e_0$. Lemma 1 follows if we combine (4.6), (4.7), (4.8) and (4.9).
 \end{proof}

 \begin{lemma}{
\begin{equation}
\begin{aligned}
 c_{2m}= \int Q(R^{(m-2)}_{uv})+ \sum\limits_{i=-2}^{m-4}h_{m,i}\int R^{(m+i)}_{11} +C_m\int Ric^{(2m-2)}+ Constant
\end{aligned}
\end{equation}
 where $Q$ is a negative definite quadratic form, $h_{m,i}$ are constants and $C_m$ is a negative constant.}
\end{lemma}
 \begin{proof}

  By the induction hypothesis and (2.3), we have

\begin{equation}
\begin{aligned}
C^u_{u,2m+1}&=\sum\limits_{k+j=2m-1,w}\frac{C^w_{u,k}R^{(j)}_{uw}}{j!(2m+1)2m}\\&=\frac{1}{(2m+1)2m}(\sum\limits_w(
  \frac{R^{(m-2)}_{uw}C^w_{u,m+1}}{(m-2)!}\\&+\sum\limits_{j=m-1}^{2m-2}B_{j,m,w,u}R^{(j)}_{uw})+R_{uu}C^u_{u,2m-1})
 \end{aligned}
\end{equation}

  where $B_{j,m,w,u}$ are constants.
For $i \leq m$, we have

\begin{equation}
C^v_{u,m+i}=\sum\limits_{j=m-2}^{m+i-3}d_{m,i,j,w,u}R^{(j)}_{uw}+ Constant
\end{equation}
where $d_{m,i,j,w,u}$ are constants. In particular, we have
\begin{equation}
   C^{v}_{u,m+1}=\sum\limits_{k+j=m-1, w}C^w_{u,k}\frac{R^{(j)}_{vw}}{j!m(m+1)}= \frac{1}{m(m+1)}(\frac{R^{(m-2)}_{vu}}{(m-2)!}+C^{v}_{u,m-1}R_{vv}).
\end{equation}

  By the induction hypothesis,
\begin{equation}
   \sum\limits_uR^{(m-2)}_{uu}=-Ric^{(m-2)}=0.
 \end{equation}
 Therefore
\begin{equation}
\begin{aligned}
   \sum\limits_u(R^{(m-2)}_{uu})^2&=(\sum\limits_uR^{(m-2)}_{uu})^2-2\sum\limits_{u<v}R^{(m-2)}_{uu}R^{(m-2)}_{vv}\\&=
   -2\sum\limits_{u<v}R^{(m-2)}_{uu}R^{(m-2)}_{vv}.
\end{aligned}
\end{equation}

Inserting $(4.11),(4.12),(4.13)$ in $(4.5)$, we find
\begin{equation}
  c_{2m}=\int Q(R^{(m-2)}_{uv})+ \sum\limits_{i=-2}^{m-2}\int \sum\limits_{u, v}h_{m,i,u,v}R^{(m+i)}_{uv} + Constant.
\end{equation}

Now we prove that $Q$ is negative definite. Let us check each term in (4.5).

By (4.13) ,the term $\frac{1}{2}\sum\limits_{u,v}(C^v_{u,m+1})^2$ in (4.5) contributes to the quadratic term
\begin{equation}
\sum\limits_{u,v}\frac{1}{2m^2(m+1)^2((m-2)!)^2}(R_{uv}^{(m-2)})^2.
\end{equation}

 The term $2\sum\limits_{u<v}C^u_{u,m+1}C^v_{v,m+1}$ contributes to the quadratic term
\begin{equation}
\sum\limits_{u<v}\frac{2}{m^2(m+1)^2((m-2)!)^2}R_{uu}^{(m-2)}R_{vv}^{(m-2)}.
\end{equation}

By (4.15), it could be written as
\begin{equation}
 -\frac{1}{m^2(m+1)^2((m-2)!)^2}\sum\limits_u(R^{(m-2)}_{uu})^2.
\end{equation}

By (4.11) and (4.13), the term $\sum\limits_{u}C_{u,2m+1}^u$ contributes to the quadratic term
\begin{equation}
\sum\limits_{u,v}\frac{1}{2m^2(m+1)(2m+1)((m-2)!)^2}(R_{uv}^{(m-2)})^2.
\end{equation}

The term $-2\sum\limits_{u<v}C^v_{u,m+1}C^u_{v,m+1}$ contributes to the quadratic term
\begin{equation}
-\sum\limits_{u<v}\frac{2}{m^2(m+1)^2((m-2)!)^2}(R_{uv}^{(m-2)})^2.
\end{equation}

The term $-\frac{1}{2}(\sum\limits_{u}C^{u}_{u,m+1})^2$ is obvious semi-negative definite.

Combine (4.17), (4.18), (4.19), (4.20) and (4.21), it follows that the quadratic form in (4.10) is negative definite.

\bigskip

Consider the linear terms in (4.16).
By the induction hypothesis, the coefficients $h_{m, i, u, v}$ are unchanged if we take a unitary
transformation keeping the direction $e_0$ fixed. Comparing the coefficients of the linear order terms, we see $h_{m, i, u, v}$ = 0 if $u \neq v$; $h_{m, i, u, u} = h_{m, i, v, v}$ if $u \neq e_1$ and $v \neq e_1$. Therefore, the linear
 terms $h_{m, i, u, u}R^{(m+i)}_{uu}$ could be absorbed in $Ric^{(m+i)}$ with the terms $-h_{m, i}R^{(m+i)}_{11}$ left. Also note that
by induction hypothesis, $Ric^{(l)} = 0$ for $0 < l \leq 2m-3$($Ric^{(2m-3)}$ vanishes as the Ricci curvature attains its minimum at $p$). Finally, one verifies
 that $\sum\limits_{u}C_{u,2m+1}^u$ is the only term in (4.5) that has contribution to $R_{uv}^{(2m-2)}$. Therefore the linear terms in (4.16) could be written as $\sum\limits_{i=-2}^{m-4}h_{m,i}\int R^{(m+i)}_{11} + C_m\int Ric^{(2m-2)}$. From (4.11), it is simple to check that $C_m$ is negative.

 \end{proof}

\bigskip

  By the induction hypothesis and that the Ricci curvature attains its minimum at $p$, we have $Ric^{(2m-2)}\geq 0$. It follows from lemma 2 that
\begin{equation}
c_{2m} \leq  \sum\limits_{i=-2}^{m-4}h_{m,i}\int R^{(m+i)}_{11} + Constant.
\end{equation}

We would like to prove that the linear terms $\int R^{(m+i)}_{11}$ vanish for $-2\leq i\leq m-4$. Note that by symmetry, if $m+i$ is odd, the integral equals $0$.
 Let us deal with case when $m+i$ is even. We shall check when $i=m-4$. Other cases are similar. Let
\begin{equation}
A=-\frac{1}{4}\int R^{(2m-4)}_{11}.
\end{equation}

Set up an orthonormal frame $\{f_i\}$ at $p$ such that $Jf_{2j}=f_{2j+1}, Jf_{2j+1}=-f_{2j}$ for $0\leq j\leq n-1$. Letting $\beta_j= \frac{1}{2} (f_{2j}-\sqrt{-1}f_{2j+1})$,  in a small neighborhood of $p$, we parallel transport the frame along each geodesic through $p$.
Suppose
\begin{equation}
  e_0= \sum\limits_{j=0}^{n-1}(z_j\beta_j+\overline{z_j}{\overline{\beta_j}}).
\end{equation}
\begin{lemma}{ Under the assumption of the induction in proposition 1,  $Rm^{(\lambda)}=0$ at $p$
 for $1\leq \lambda \leq m-3$, where $Rm^{(\lambda)}$ denotes any covariant derivative of the curvature tensor with order $\lambda$ at $p$.}
\end{lemma}
\begin{proof}
We use induction.
  If $\lambda = 0$, lemma 3 automatically holds since there is nothing to prove.
  Suppose lemma 3 holds for $k<\lambda$.
  For $k=\lambda$, we plug (4.24) in $R^{(\lambda)} _{uv}$.
\begin{claim}  { We can commute the covariant derivatives for $R^{(\lambda)} _{uv}.$}
 \end{claim}

\begin{proof}
To prove claim 2, we only need to consider the case $\lambda \geq 2$.
   By the induction hypothesis of lemma 3, the covariant derivatives of the curvature tensor vanish up to order $\lambda -1$ at $p$.
If $\lambda > 3$, claim 2 follows from the ricci identity.
Now suppose $\lambda = 2$.
 By ricci identity, the difference of commuting the covariant derivatives is a function of the curvature tensor.
Note that the curvature tensor at $p$ is the same as the complex space form. We complete the proof for $\lambda = 2$.
\end{proof}
We insert (4.24) in $R^{(\lambda)} _{Je_0Je_0}$. By claim 2 and Bianchi identities,
 $R^{(\lambda)} _{Je_0Je_0}$ becomes a polynomial with variables $z_j, \overline{z_j}$. The coefficients of the polynomial are exactly all the covariant derivatives of $Rm$ at $p$ with order $\lambda$. According to the assumption of lemma 3, $R^{(\lambda)} _{Je_0Je_0}$ is identically 0 for all $e_0$. Therefore, the coefficients of the polynomial are all 0.
This completes the induction of lemma 3.
\end{proof}

\begin{lemma} {Under the assumption of the induction in proposition 1, $A$ could be written
 as $\sum\limits_{i=1}^{m-2}g_{i,m}\Delta^{i}s$ where $s$ denotes the scalar curvature, $g_{i,m}$ are constants
depending only on $n, m, i$.}
\end{lemma}
\begin{proof}
Define $X = \frac{1}{2}(e_0 - \sqrt{-1}Je_0)$, then $A= \int R_{X\overline{X}X\overline{X},e_0e_0...e_0}$ where the number of $e_0$ is $2m-4$. Plugging (4.24) in it, after the integration, we find
\begin{equation}
  A=\sum\limits_{\alpha_1\alpha_2...\alpha_{2m}}(\int \alpha_1\alpha_2...\alpha_{2m}) R_{\alpha_1\alpha_2\alpha_3\alpha_4,\alpha_5....\alpha_{2m}}
\end{equation}
  where $\alpha_i$ is \{$z_j$\} or \{$\overline{z_k}$\} for $0\leq j, k\leq n-1$, $\alpha_1, \alpha_3\in \{z_j\}$, $\alpha_2, \alpha_4\in \{\overline{z_k}$\}. Under the subscript of $R$, $z_j$ stands for $\beta_j$, $\overline{z_k}$ stands for $\overline{\beta_k}$.

   From the expression of (4.25), we see that $z_i,\overline{z_i}$ must all go in pairs in the sequence $\alpha_1\alpha_2..\alpha_{2m}$, otherwise the integral $\int \alpha_1\alpha_2...\alpha_{2m}$ equals $0$. Switching the covariant derivatives in (4.25), using K\"{a}hler identities, we can rearrange (4.25) as
\begin{equation}
A= \sum\limits_{I_1, I_2,... I_n} C_{I_1I_2..I_n}R_{I_1I_2...I_n}+ B
\end{equation}
 where the symbol $I_j$ denotes $z_j\overline{z_j}..z_j\overline{z_j}$; subscripts after the fourth subscript of $R$ denote the covariant derivatives; $C_{I_1I_2..I_n}$ are the coefficients; $\sum\limits_j|I_j|= 2m$; $B$ is the combination of covariant derivatives of $Rm$ with lower order. From (4.23), we see that the coefficients $C_{I_1I_2..I_n}$ in $(4.26)$ are unitary invariants. For fixed $I_3, I_4,..I_n$, let $d=|I_1| + |I_2|$. Denote $C_{I_1I_2..I_n}$ by $C_p$ where $0 \leq |I_1|=p \leq d$. We want to find the relations of \{$C_p$\}.
Take a unitary transformation:

    $ \widetilde{\beta_i}=\beta_i$ for $i \neq 1,2$;
   $\beta_1=\cos\theta\widetilde{\beta_1}+\sin\theta\widetilde{\beta_2}$;
   $\beta_2=-\sin\theta\widetilde{\beta_1}+\cos\theta\widetilde{\beta_2}$.

   Insert the unitary transformation above in  (4.26), the new coefficient $\tilde{C_d}$ becomes $\sum\limits_{p=0}^{d}C_p\cos^{2p}\theta\sin^{2(d-p)}\theta$. Therefore we have:
\begin{equation}
\sum\limits_{p=0}^{d}C_p\cos^{2p}\theta\sin^{2(d-p)}\theta=C_d=C_d(\cos^2\theta+\sin^2\theta)^d.
\end{equation}
\begin{claim}{ $C_p = C_d\binom dp$}

\end{claim}

\begin{proof}

  Divide by $cos^{2d}{\theta}$ on both sides, (4.27) becomes
 \begin{equation}\nonumber
 \sum\limits_{p=0}^{d}C_p\tan^{2(d-p)}\theta=C_d=C_d(1+\tan^2\theta)^d.
 \end{equation}

 Since $\theta$ is arbitrary, claim 3 follows.
\end{proof}

   By claim 3, $\frac{C_p}{C_d}=\binom dp$. Since we can substitute any index $u, v$ for $1, 2$, the ratio of all coefficients
in (4.26) are determined. Note that to get the relations between $C_p$, we only use the condition that the form (4.23) is unitary invariant. Since $\Delta^{m-2}s$ is also unitary invariant with respect to the frame, we can write it in the form as (4.26). By the same argument, the ratio of all coefficients of $\Delta^{m-2}s$ are the same as (4.26).
    It follows that the term $\sum\limits_{I_1, I_2,... I_n} C_{I_1I_2..I_n}R_{I_1I_2...I_n}$ in (4.26) equals
$C(m,n)\Delta^{(m-2)}s$ modulo lower order covariant derivatives, where $C(m, n)$ is a constant depending only on $m, n$.

\bigskip

   Now we make an important observation. From the Ricci identity,
 $R_{i_1\overline{i_2}....i_p\alpha\beta i_{p+3}..i_{2m}}-R_{i_1\overline{i_2}...i_p\beta\alpha i_{p+3}..i_{2m}}$ is the sum of
   $(RmRm^{(p-4)})_{,i_{p+3}..i_{2m}}$. By lemma 3, $Rm^{(\lambda)}=0$ for $1\leq\lambda\leq m-3$. It follows
that $(RmRm_{,i_{5}...i{p}})_{,i_{p+3}..i_{2m}}$ can be expanded as a linear combination of the covariant derivatives
of curvature tensor. Therefore $A-C(m, n)\Delta^{(m-2)}s$ can be written as a linear combination of the covariant derivatives
of the curvature tensor with the highest order $2m-6$.  Furthermore it is unitary invariant since the curvature tensor is unitary
 invariant at $p$. By recursive arguments, we complete the proof of lemma 4.
\end{proof}

  From the induction in proposition 1, $Ric^{(l)}=0$ for $1\leq l\leq 2m-4$.
  Integrating with respect to the unit sphere in $T_pM$, by similar arguments as in the proof of lemma 4, we find that for $l$ even,
\begin{equation}
0=\int Ric_{e_0e_0,e_0e_0...e_0}=\sum\limits_{k=1}^{\frac{l}{2}}C_{l,k}\Delta^{k}s
\end{equation}
where the order of the covariant derivative above is $l$. It is straightforward to check that the highest order coefficient $C_{l,\frac{l}{2}}$ is not equal to 0. Then by a recursive argument, $\Delta^{k}s=0$ at $p$ for $1\leq k\leq m-2$. Combine this with lemma 4, it follows that $A=0$. Similarly all linear terms in (4.10) vanish. Therefore, under the induction hypothesis in proposition 1, in order that $c_{2m}$ in (4.10) achieves the maximum, $Ric^{(2m-2)}=0$ and $R^{(\lambda)}_{uv}=0$ for $1\leq\lambda\leq m-2$. This is exactly the case of the complex space form. Therefore we complete the induction in proposition 1. As a byproduct, we proved conclusion 2 in proposition 1.
The proof of proposition 1 is complete. \qed

\section{\bf{The proof of theorem 4}}

    Under the assumption of proposition 2, using the same argument as in the last section, we find that $W^{(2m+1)}$ is a linear combination of $\int R^{(m+i)}_{11}$($1 \leq i \leq m-3$)(the terms with order greater than $2m-3$ could be absorbed in $Ric^{(m+i)}$ to vanish). Similar as the proof of lemma 4, $W^{(2m+1)}$ is equal to $0$. This completes the proof of proposition 2.

Consider two cases below:

1. All coefficients of the power series of $W$ are equal to that of the complex space form. Follow proposition 1, all covariant derivatives of the curvature tensor at $p$ are the same as the complex space form. Since the metric is real analytic, we conclude that near $p$, the manifold is isometric to the complex space form.

2. There is a $i_0 \geq 1$ such that for all $i < i_0$, the coefficients of the power series of $W$ are equal to that of the complex space form, but the $i_0$th coefficient is less than that of the complex space form.
Checking the power series of $\frac{W'}{W}$ at $p$, we find that for sufficiently small $r$, $\frac{W'}{W}$ is less than that of the complex space form. Follow the definition of $W$, for small $r$,
 $$\frac{\int_{\partial B_p(r)}\Delta r}{A(\partial B_p(r))} = \frac{\int (\sqrt{det<J_u,J_v>})'}{\int \sqrt{det<J_u,J_v>}} < \Delta_{N_K}r(r).$$

 The proof of theorem 4 is complete. \qed

\section{\bf{An example}}

In this section we give an example showing that the analogous Laplacian comparison theorem is not true on K\"{a}hler manifolds when the Ricci curvature is bounded from below by a nonzero constant. The example is in dimension 2. For higher dimensions, the construction is similar.

 Identify $\mathbb{R}^4$ with $\mathbb{C}^2$ in the usual way.
The corresponding almost complex structure $J$ is given by $J\frac{\partial}{\partial x_1} = \frac{\partial}{\partial x_2}, J\frac{\partial}{\partial x_2} = -\frac{\partial}{\partial x_1}, J\frac{\partial}{\partial x_3} = \frac{\partial}{\partial x_4}, J\frac{\partial}{\partial x_4} = -\frac{\partial}{\partial x_3}$.

Given a small ball near the origin of $\mathbb{C}^2$, define the function $f$ to be
\begin{displaymath}
\begin{aligned}
f&= |z_1|^2+|z_2|^2+a|z_1|^4+8a|z_1|^2|z_2|^2+a|z_2|^4+\frac{8}{3}a^2|z_1|^6+\\&28
a^2|z_1|^4|z_2|^2+28a^2|z_1|^2|z_2|^4+\frac{8}{3}a^2|z_2|^6+p(|z_1|,|z_2|)
\end{aligned}
\end{displaymath}
where $a$ is a nonzero constant and $p$ is a homogeneous polynomial of degree 8 which will be determined later.

We define $$\omega = \frac{\sqrt{-1}}{2}\partial\overline{\partial}f = \frac{\sqrt{-1}}{2}\sum\limits_{i,j} g_{i\overline{j}}dz_i \wedge d\overline{z_j}.$$
It is straightforward to check that $\omega$ defines a K\"{a}hler metric $g$ if the ball is sufficiently small (note that the metric is not complete).

Direct computation gives
\begin{displaymath}
\begin{aligned}
g_{1\overline{1}}&=1+4a|z_1|^2+8a|z_2|^2+24a^2|z_1|^4+112a^2|z_1|^2|z_2|^2+28a^2|z_2|^4\\&+O((|z_1|+|z_2|)^6));
\end{aligned}
\end{displaymath}

\begin{displaymath}
\begin{aligned}
g_{2\overline{2}}&=1+4a|z_2|^2+8a|z_1|^2+24a^2|z_2|^4+112a^2|z_1|^2|z_2|^2+28a^2|z_1|^4\\&+O((|z_1|+|z_2|)^6));
\end{aligned}
\end{displaymath}
\begin{displaymath}
\begin{aligned}
g_{1\overline{2}}=8a\overline{z_1}z_2+56a^2z_1\overline{z_1}^2z_2+56a^2\overline{z_1}z_2^2\overline{z_2}+O((|z_1|+|z_2|)^6)).
\end{aligned}
\end{displaymath}
Therefore
\begin{displaymath}
\begin{aligned}
det(g_{i\overline{j}})&=g_{1\overline{1}}g_{2\overline{2}}-|g_{1\overline{2}}|^2\\&=(1+4a|z_1|^2+8a|z_2|^2+24a^2|z_1|^4+112a^2|z_1|^2|z_2|^2+28a^2|z_2|^4)
\\&(1+4a|z_2|^2+8a|z_1|^2+24a^2|z_2|^4+112a^2|z_1|^2|z_2|^2+28a^2|z_1|^4)\\&-|8a\overline{z_1}z_2+56a^2z_1\overline{z_1}^2z_2+56a^2\overline{z_1}z_2^2\overline{z_2}|^2
+O((|z_1|+|z_2|)^6)\\&=1+12a(|z_1|^2+|z_2|^2)+84a^2(|z_1|^4+|z_2|^4)+240a^2|z_1|^2|z_2|^2\\&+O((|z_1|+|z_2|)^6)).
\end{aligned}
\end{displaymath}
Using $log(1+x)=x-\frac{1}{2}x^2+O(x^3)$, we have
\begin{displaymath}
\begin{aligned}
Ric + 12ag &= \partial\overline{\partial}(-log (detg_{i\overline{j}})+ 12af)
=\partial\overline{\partial}(O((|z_1|+|z_2|)^6)).
\end{aligned}
\end{displaymath}
Therefore $Ric + 12ag$ vanishes up to order 3 at the origin.
Moreover, if we choose the function $p$ to be $-\lambda(|z_1|^8+|z_2|^8+8(|z_1|^6|z_2|^2+|z_1|^2|z_2|^6))$, after a direct computation,
 \begin{displaymath}
\begin{aligned}
 Ric + 12ag =  \partial\overline{\partial}(24\lambda(|z_1|^2+|z_2|^2)^3+O((|z_1|+|z_2|)^6)
 \end{aligned}
\end{displaymath}
where the term $O((|z_1|+|z_2|)^6)$ does not depend on $\lambda$.
If $\lambda$ is sufficiently large,
$Ric + 12ag \geq 0$ near the origin. Set $K = -12a$. Thus, near the origin, $Ric \geq K$.
By direct computation, at the origin, $R_{1212}=R_{1313}=R_{1414}=4a$; $R_{1u1v}=0$ if $u \neq v$.
Combining this with the fact that the second derivatives of the Ricci tensor vanish at the origin,
after a slight computation, we find that the fourth order term of (3.3) is greater than that of the complex
space form if $e_0 = \frac{\partial}{\partial x_1}$. So when $r$ is very small, along the geodesic with initial
direction $\frac{\partial}{\partial x_1}$ at the origin, $\sqrt{det<J_u,J_v>}$ is greater than that of the complex
space form. Since $\Delta r = \frac{\partial \log\sqrt{det<J_u,J_v>}}{\partial r}$, it follows that the pointwise
Laplacian comparing with the complex space forms is not true for K\"ahler manifolds.

  \end{document}